\documentclass[11pt,reqno]{amsart}
\usepackage{amsmath}
\usepackage{amssymb}
\usepackage{psfrag}
\usepackage{graphicx}
\usepackage{float}

\newtheorem{theorem}{Theorem}
\newtheorem{prop}{Proposition}
\newtheorem{lemma}{Lemma}
\newtheorem{cor}{Corollary}

\theoremstyle{definition}

\newtheorem{defn}{Definition}

\newtheorem{remark}{Remark}

\def\I{\mathcal I}
\def\i{\underline i}
\def\j{\underline \iota}
\def\D{\frak D}

\def\l{\mathcal L}

\def\N{\mathbb N}
\def\R{\mathbb R}

\def\CC{\mathbb C}

\def\phi{\varphi}
\def\rho{\varrho}
\def\epsilon{\varepsilon}
\def\i{\underline i}

\begin{document}

\newcounter{algnum}
\newcounter{step}
\newtheorem{alg}{Algorithm}

\newenvironment{algorithm}{\begin{alg}\end{alg}}

\title[Effective bounds on Hausdorff dimension]{Rigorous effective bounds on the Hausdorff dimension of 
continued fraction Cantor sets: a hundred decimal digits for the dimension of $E_2$}

\author{O. Jenkinson \& M. Pollicott}

\begin{abstract}
We prove that the algorithm of \cite{jp} for approximating the Hausdorff dimension of dynamically defined Cantor sets, using periodic points of the underlying dynamical system, can be used to establish completely rigorous high accuracy bounds on the dimension. The effectiveness of these rigorous estimates is illustrated for Cantor sets consisting of continued fraction expansions with restricted digits. For example the Hausdorff dimension of the set $E_2$ (of those reals whose continued fraction expansion only contains digits 1 and 2) can be rigorously approximated, with an accuracy of over 100 decimal places, using points of period up to 25.

The method for establishing rigorous dimension bounds involves
the holomorphic extension of mappings associated to the allowed continued fraction digits,
 an appropriate disc which is contracted by these mappings, 
and an associated transfer operator acting on the 
Hilbert Hardy space of analytic functions on this disc. We introduce
methods for  
rigorously bounding the approximation numbers for the transfer operators, showing that this leads to effective estimates on the Taylor coefficients of the 
associated determinant, and hence to explicit bounds on the Hausdorff dimension. 
\end{abstract}

\maketitle

\section{Introduction}

For a finite subset $A\subset \N$, let $E_A$ denote the set of all $x\in(0,1)$ such that the digits
$a_1(x), a_2(x),\ldots$ in the continued fraction expansion 
$$
x = 
[a_1(x), a_2(x), a_3(x), \ldots ]
= 
\frac{1}{a_1(x) + \frac{1}{a_2(x) + 
\frac{1}{a_3(x) + \cdots}
}}
$$
all belong to $A$.
Sets of the form $E_A$ are said to be of \emph{bounded type} (see e.g.~\cite{kontorovich, shallit});
in particular they are Cantor sets, and study of their Hausdorff dimension has attracted significant attention.

Of particular interest have been the sets $E_n=E_{\{1,\ldots,n\}}$,
with $E_2=E_{\{1,2\}}$ the most studied of these, serving as a test case for various general methods
of approximating Hausdorff dimension. 
Jarnik \cite{jarnik} showed that $\text{dim}(E_2)>1/4$,
while Good \cite{good} improved this to $0.5306  < \text{dim}(E_2)<0.5320$,
Bumby \cite{bumby2}
showed that $0.5312 < \text{dim}(E_2) < 0.5314$,
Hensley \cite{hensley1989} showed that
$0.53128049$ $< \text{dim}(E_2) < 0.53128051$,
while
Falk \& Nussbaum \cite{falknussbaum}\footnote{This preprint has been split into the two articles
\cite{falknussbaum1} and \cite{falknussbaum2}, with \cite{falknussbaum1} containing the approximation to
$\text{dim}(E_2)$.} 
rigorously justified the first 8 decimal digits of $\text{dim}(E_2)$,
proving that
$0.531280505981423 \le \text{dim}(E_2) \le 0.531280506343388\,.$
A common element in the methods \cite{bumby2, falknussbaum, hensley1989} is 
the study of a transfer operator, while for the higher accuracy estimates 
\cite{falknussbaum, hensley1989} there is some element of
computer-assistance involved in the proof.

In \cite{jp} we outlined a different approach to approximating the Hausdorff dimension of bounded type sets,
again using a transfer operator, but exploiting the real analyticity of the maps defining continued fractions to
consider the determinant $\Delta$ of the operator, and its
 approximation in terms of periodic points\footnote{The periodic points are precisely those numbers in $(0,1)$
 with periodic continued fraction expansion, drawn from digits in $A$.
 The reliance on periodic points renders the method \emph{canonical}, inasmuch as it does not involve any arbitrary choice of coordinates or partition of the space.}
  of an underlying dynamical system.
While some highly accurate \emph{empirical} estimates of Hausdorff dimension were given, for example a 25 decimal digit approximation to $\text{dim}(E_2)$, these were not rigorously justified. Moreover, although the 
algorithm was proved to generate a sequence of approximations $s_n$ to the Hausdorff dimension
(depending on points of period up to $n$),
with convergence rate faster than any exponential, 
the derived error bounds were sufficiently conservative (see Remark \ref{conservative} below) that it was unclear
whether  they could be combined with 
the computed approximations to yield any effective \emph{rigorous} estimate.

In the current paper we investigate the possibility of sharpening the approach of \cite{jp} 
so as to obtain rigorous computer-assisted estimates on $\text{dim}(E_A)$, with particular focus on $E_2$.
There are several ingredients in this sharpening.
The first step is to locate a disc $D$ in the complex plane with the property that
the images of $D$ under the mappings $T_n(z)=1/(z+n)$, $n\in A$, are contained in $D$.
It then turns out to be preferable to consider the transfer operator as acting on a \emph{Hilbert} space of analytic functions on $D$, rather than the Banach space of \cite{jp};
this facilitates an estimate on the Taylor coefficients of $\Delta$ in terms of the \emph{approximation numbers}
(or \emph{singular values}) of the operator, which is significantly better than those bounds derived from Banach space methods. 
The specific Hilbert space used is Hardy space, consisting of those analytic functions on the disc which extend as $L^2$
functions on the bounding circle.
The contraction of $D$ by the mappings $T_n(z)=1/(n+z)$, $n\in A$,
prompts the introduction of the \emph{contraction ratio}, which captures
the strength of this contraction, and leads to estimates on the convergence of
the approximations to the Hausdorff dimension.
The $n^{th}$ Taylor series coefficient of $\Delta$ can be expressed in terms of periodic points of period up to $n$, and for sufficiently small $n$ these can be evaluated
exactly,  to arbitrary precision.
For larger $n$, we show it is advantageous to
obtain two distinct types of upper bound on the Taylor coefficients: we refer to these
as the \emph{Euler bound} and the \emph{computed Taylor bound}.
The Euler bound is used for all sufficiently large $n$, while the
computed Taylor bound is used for a finite intermediate range of $n$
corresponding to those Taylor coefficients which are deemed to be computationally inaccessible,
but 
where the Euler bound is insufficiently sharp.
Intrinsic to the definition of the computed Taylor bounds is the sequence of
\emph{computed approximation bounds}, which we introduce as 
computationally accessible upper bounds on
the approximation numbers 
of the transfer operator.

As an example of the effectiveness of the resulting method we rigorously justify the 
first 100 decimal digits\footnote{The choice of 100 decimal digits in the present article is motivated by a number of factors.
On the one hand 100 is considered a particularly round number, and an order of magnitude larger than  the number of decimal digits obtained (even non-rigorously) for the dimension of $E_2$ in previous works.
On the other hand, readily available computer resources (namely, a program written in Mathematica running on a modestly equipped laptop) performed the necessary calculations, in particular the high accuracy evaluation of points of period up to 25, in a reasonable timeframe (approximately one day), and it turns out that this choice of maximum period is sufficient to rigorously justify 100 decimal digits.}
of the Hausdorff dimension of $E_2$, 
thereby improving on the rigorous estimates in \cite{bumby2, falknussbaum, good, hensley1989, jarnik}.
Specifically, we prove (see Theorem \ref{e2theorem}) that 
\begin{align*}
\text{dim}(E_2) =
0.&53128050627720514162446864736847178549305910901839 \\
& 87798883978039275295356438313459181095701811852398\ldots \,,
\end{align*}
 using the 
  periodic points of period up to 25.


\section{Preliminaries}

In this section we collect a number of results (see also \cite{jp}) which underpin our algorithm for approximating Hausdorff dimension.

\subsection{Continued fractions}

Let $E_A$ denote the set of all $x\in(0,1)$ such that the digits
$a_1(x), a_2(x),\ldots$ in the continued fraction expansion 
$$
x = 
[a_1(x), a_2(x), a_3(x), \ldots ]
= 
\frac{1}{a_1(x) + \frac{1}{a_2(x) + 
\frac{1}{a_3(x) + \cdots}
}}
$$
all belong to $A$.
For any $i\in\N$ we define the map $T_i$ by
$$
T_i(x)=\frac{1}{i+x}\,,
$$
and for a given $A\subset \N$, the collection $\{T_i:i\in A\}$ is referred to as the corresponding
\emph{iterated function system}.
Its \emph{limit set}, consisting of limit points of sequences
$T_{i_1}\circ \cdots\circ T_{i_n}(0)$, where each $i_j\in A$, is precisely the set $E_A$.

Every set $E_A$ is invariant under the \emph{Gauss map} $T$, defined by
$$
T(x) = \frac{1}{x} \pmod 1\,.
$$

\subsection{Hausdorff dimension}

For a set $E \subset \R$,
define
$$
H_\epsilon^\delta(E) = \inf
\left\{\sum_i\text{diam}(U_i)^\delta : \text{$\mathcal{U} = \{U_i\}$ is an open cover of $E$ such that each $\text{diam}(U_i) \leq \epsilon$}\right\}\,,
$$
and set
$H^\delta(E) =\lim_{\epsilon \to 0} H_\epsilon^\delta(E)$.
The \emph{Hausdorff dimension} $\text{dim}(E)$
is then defined as
$$
\text{dim}(E) = \inf\{\delta \hbox{ : } H^\delta(E) = 0\} \,.
$$

\subsection{Pressure formula}

For a continuous function $f: E_A \to \R$,
its \emph{pressure} $P(f)$ 
is given by
$$
P(f) = \lim_{n\to +\infty}\frac{1}{n}
\log \left(\sum_{T^nx=x \atop x\in E_A}
e^{f(x) + f(Tx) + \ldots + f(T^{n-1}x)}
\right)\,,
$$
and if
$f = -s\log|T'|$ then we have the following implicit characterisation of the Hausdorff
dimension of $E_A$ (see \cite{bedford, bowen, falconer, mauldinurbanski}):

\begin{lemma}\label{pressurelemma}
The function $s\mapsto P(-s\log |T'|)$ is strictly decreasing,
with a unique zero at $s=\text{dim}(E_A)$.
\end{lemma}

\subsection{Transfer operators}

For a given $A\subset \N$, and $s\in\R$, the \emph{transfer operator} $\l_{A,s}$, defined by
$$
\l_{A,s} f (z) = \sum_{i\in A} \frac{1}{(z+i)^{2s}}\, f\left(\frac{1}{z+i}\right) \,,
$$
preserves various natural function spaces, for example the Banach space of Lipschitz functions on $[0,1]$.
On this space it has a simple positive eigenvalue $e^{P(-s\log|T'|)}$, which is the unique eigenvalue whose modulus equals 
its spectral radius,
thus by Lemma \ref{pressurelemma} 
the Hausdorff dimension of $E_A$ is the unique value $s\in\R$ such that $\l_{A,s}$ has spectral radius equal to 1.

\subsection{Determinant}
The \emph{determinant} for $\l_{A,s}$ is the entire function defined for $z$ of sufficiently small modulus\footnote{The power series $\sum_{n=1}^\infty \frac{z^n}{n} \text{tr}(\l_{A,s}^n)$, and hence the expression (\ref{determinantformula}), is convergent for $|z|< e^{-P(-s\log|T'|)}$.}
by
\begin{equation}\label{determinantformula}
\Delta(z,s) = \exp -\sum_{n=1}^\infty \frac{z^n}{n} \text{tr}(\l_{A,s}^n) \,,
\end{equation}
and for other $z\in\mathbb{C}$ by analytic continuation;
here the trace $\text{tr}(\l_{A,s}^n)$ is given (see \cite{jp, ruelle}) by 
\begin{equation}\label{traceformula}
\text{tr}(\l_{A,s}^n) = \sum_{\i\in A^n}  \frac{|T_{\i}'(z_{\i})|^s}{1-T_{\i}'(z_{\i})} 
 = \sum_{\i\in A^n}  \frac{\prod_{j=0}^{n-1}T^j(z_{\i})^{2s}}{1-(-1)^n\prod_{j=0}^{n-1}T^j(z_{\i})^2 } 
\,,
\end{equation}
where the point $z_{\i}$, which has period $n$ under $T$, is the unique fixed point of the 
$n$-fold composition $T_{\i} = T_{i_1}\circ T_{i_2}\circ\cdots\circ T_{i_n}$.

When acting on a suitable space of holomorphic functions, the eigenvalues of $\l_{A,s}$ are precisely the reciprocals
of the zeros of its determinant.
In particular, the zero of minimum modulus for $\Delta(s,\cdot)$ is $e^{-P(-s\log|T'|)}$, so the Hausdorff dimension
of $E_A$ is characterised as the value of $s$ such that 1 is the zero of minimum modulus of $\Delta(s,\cdot)$.

In fact we shall later show that, when $\l_{A,s}$ acts on such a space of holomorphic functions, its 
approximation numbers decay at an exponential rate (see Corollary \ref{ks}), so that $\l_{A,s}$ belongs to an exponential class (cf.~\cite{bandtlow, bj}) and is in particular a trace class operator, from which the existence and above
properties of trace and determinant follow (see \cite{simon}).

As outlined in \cite{jp}, this suggests the possibility of expressing $\Delta(z,s)$ as a power series
$$
\Delta(z,s) = 1 + \sum_{n=1}^\infty \delta_n(s) z^n\,,
$$
then 
defining $\D$ by
$$
\D(s) := \Delta(1,s) =1 + \sum_{n=1}^\infty \delta_n(s) \,.
$$
The function $\D$ is an entire function of $s$ (see \cite{jp}),
and solutions $s$ of the equation
\begin{equation}\label{propereqn}
0 = 1+\sum_{n=1}^\infty \delta_n(s) = \D(s)
\end{equation}
have the property that the value 1 is an eigenvalue for $\l_{A,s}$;
in particular,
the unique zero of $\D$ in the interval $(0,1)$ is precisely $\text{dim}(E_A)$,
being the unique value of $s$ for which 1 is the eigenvalue of maximum modulus for $\l_{A,s}$.

As a result of the trace formula (\ref{traceformula}), the coefficients $\delta_n(s)$ are computable\footnote{By this we mean that for a given $s$, the $\delta_n(s)$ are computable exactly, to arbitrary precision.}
in terms of the periodic points of $T|_{E_A}$ of period no greater than $n$, so for some suitable
$N\in\N$, chosen so that $\delta_1(s),\ldots, \delta_N(s)$ can be computed to a given precision in reasonable time,
we can define $\D_N$ by
\begin{equation}\label{DNdefn}
\D_N(s) := 1+\sum_{n=1}^N \delta_n(s)\,.
\end{equation}
A solution to the equation
\begin{equation}\label{truncatedequation}
\D_N(s) = 0 
\end{equation}
will be an approximate solution to (\ref{propereqn}), where the quality of this approximation will be related to the 
smallness of the discarded tail
\begin{equation}\label{discardedtail}
\sum_{n=N+1}^\infty \delta_n(s)\,.
\end{equation}
In particular, any rigorous estimate of the closeness of a given approximate solution $s_N$ of (\ref{truncatedequation}) to the true Hausdorff dimension $\text{dim}(E_A)$ will require a rigorous upper bound
on the modulus of the tail (\ref{discardedtail}).

\begin{remark}\label{conservative}
In \cite{jp} we considered the set $E_2=E_{\{1,2\}}$ and, although the empirical estimates of its
Hausdorff dimension appeared convincing,
the estimate on the tail (\ref{discardedtail}) was not sharp enough to permit any effective rigorous bound.
Essentially\footnote{In \cite{jp} we actually worked with $\det(I-z\l_{A,s}^2)$ rather than
$\det(I-z\l_{A,s})$, though the methods there lead to very similar bounds for both determinants.}, the bound in \cite{jp} was
$
|\delta_n(s)| \le \epsilon_n := C K^n n^{n/2} \theta^{n(n+1)}
$
where
$
C= \gamma \prod_{r=1}^\infty(1-\gamma^r)^{-1} \approx 122979405533
$,
$
K= \frac{45}{16\pi} \approx 0.895247
$,
and
$
\theta = \left( \frac{8}{9}\right)^{1/4} \approx 0.970984
$.
Although the bounding sequence $\epsilon_n$ tends to zero, and does so at super-exponential rate $O(\theta^{n^2})$, 
the considerable inertia in this convergence 
(e.g.~the sequence increases for $1\le n\le 39$ to the value $\epsilon_{39}\approx 1.31235 \times 10^{22}$,
and remains larger than $1$ until
$n=85$)
renders the bound ineffective in practice,
in view of the exponentially increasing computation time required to calculate the $\delta_n(s)$
(as seen in this article, we can feasibly compute several million periodic points, but performing calculations involving more
 than $2^{85}$ points is out of the question).
\end{remark}

\begin{remark}
The specific rigorous approximation of dimension is performed in this article only for the set $E_2$
(see \S \ref{E2section}), corresponding to the iterated function system consisting of the maps $T_1(x)=1/(x+1)$
and $T_2(x)=1/(x+2)$.
In principle, however, it can be performed for arbitrary iterated function systems consisting of real analytic maps $T_1,\ldots, T_l$ satisfying the open set condition (i.e.~there exists a non-empty open set $U$ such that
$T_i(U)\cap T_j(U)=\emptyset$ for $i\neq j$, and $T_i(U)\subset U$ for all $i$).
In this setting the accuracy of our 
Hausdorff dimension estimate depends principally on the contractivity of the maps $T_i$ and the number $l$ of such maps, with stronger contraction and a smaller value of $l$ corresponding to increased accuracy.
Stronger contraction (as reflected by smallness of the \emph{contraction ratio} defined in \S \ref{contractionratiossubsection}) is associated with more rapid decay of the Taylor coefficients of the determinant
$\Delta(z,s)$, implying greater accuracy of the polynomial truncations, while for $l>2$ the time required to locate the points of period up to $n$ increases by a factor of roughly $(l/2)^n$ relative to the case $l=2$
(note that for \emph{infinite} iterated function systems, i.e.~$l=\infty$, our method is rarely applicable, since it is usually impossible to locate all period-$n$ points for a given $n$, though here non-rigorous approximations may be obtained by suitable approximation).
If the $T_i$ are not M\"obius maps then for practical purposes there is some minor decrease in the efficiency of our method: the compositions $T_{\underline i}$ are more highly nonlinear than in the M\"obius case, so evaluation of their fixed points typically takes slightly longer.
\end{remark}

\begin{remark}
Work of Cusick \cite{cusick1,cusick2}
on continuants with bounded digits characterised the Hausdorff dimension of $E_n=E_{\{1,\ldots,n\}}$ in terms of the abscissa of convergence of a certain Dirichlet series,
and Bumby \cite{bumby1, bumby2} showed that $0.5312 < \text{dim}(E_2) < 0.5314$.
Hensley \cite{hensley1989} obtained the bound
$0.53128049$ $< \text{dim}(E_2) < 0.53128051$ using a recursive procedure,
and in \cite[Thm.~3]{hensley1996} introduced a general approach for
approximating the Hausdorff dimension of $E_A$, obtaining in particular the empirical estimate
$\text{dim}(E_2)=
0.5312805062772051416\ldots$
\end{remark}

\section{Hilbert Hardy space, approximation numbers, approximation bounds}

In this section we introduce the Hilbert space upon which the transfer operator acts, then make the connection between
approximation numbers for the operator and Taylor
coefficients of its determinant, leading to so-called Euler bounds on these Taylor coefficients.

\subsection{Hardy space}

Let $D\subset \mathbb{C}$ be an open disc of radius $r$, centred at $c$.
The \emph{Hilbert Hardy space}
$H^2(D)$ consists of those functions $f$
which are holomorphic on $D$ and such that
$\sup_{\rho<r}
\int_0^{1} |f(c+\rho e^{2\pi it})|^2\, dt<\infty$.
The inner product on $H^2(D)$ is defined by
$(f,g)=
\int_0^{1} f(c+r e^{2\pi it}) \overline{g(c+r e^{2\pi it})}\, dt$,
which is well-defined since any element of $H^2(D)$
extends as an $L^2$ function of the boundary $\partial D$.
The norm of $f\in H^2(D)$ will be simply written as $\|f\|=(f,f)^{1/2}$.

An alternative characterisation of $H^2(D)$ 
(see e.g.~\cite{shapiro})
is as the set of functions $f$ which are holomorphic on $D$
and such that if $m_k(z)=r^{-k}(z-c)^k$ for $k\ge0$, then
$$
f = \sum_{k=0}^\infty \hat f(k)\, m_k
$$
where the sequence $(\hat f(k))_{k=0}^\infty$ is square summable.
The norm $\|f\|$ can then be expressed as
$$
\|f\|^2 = \sum_{k=0}^\infty |\hat f(k)|^2\,.
$$

\subsection{Approximation numbers}

Given a compact operator $L:H\to H$
on a Hilbert space $H$, its
$i^{th}$ \emph{approximation number} $s_i(L)$
is defined as
$$
s_i(L)=\inf\{ \| L-K \| : \text{rank}(K) \leq i-1 \} \,,
$$
so that in particular $s_1(L)=\| L \|$.

The following result exploits our Hilbert space setting,
and represents an improvement over analogous Banach space estimates
in \cite{jp} (where e.g.~a multiplicative factor $n^{n/2}$ reduces the quality of the bound on $|\delta_n(s)|$).

\begin{lemma}\label{gohberglemma}
If $\l_{A,s}:H^2(D)\to H^2(D)$, then the $n^{th}$ Taylor coefficient $\delta_n(s)$ of its determinant can be bounded by
\begin{equation}\label{Taylorapproximationbound}
|\delta_n(s)|
\leq
 \sum_{i_1 < \ldots < i_n}
\prod_{j=1}^n  s_{i_j}(\l_{A,s})\,.
\end{equation}
\end{lemma}
\begin{proof}
If $\{\lambda_n(s)\}$ is the eigenvalue sequence for $\l_{A,s}$, 
 ordered by
decreasing modulus and counting algebraic multiplicities,
then
(see e.g.~\cite[Lem.~3.3]{simon}) we have
$$\delta_n(s)
= \sum_{i_1 < \ldots < i_n}
\prod_{j=1}^n
\lambda_{i_j}(s)
\,,$$
and
\begin{equation*}
\left|\sum_{i_1 < \ldots < i_n}
\prod_{j=1}^n
\lambda_{i_j}(s)
\right|
\leq
 \sum_{i_1 < \ldots < i_n}
 \prod_{j=1}^n  s_{i_j}(\l_{A,s})
\end{equation*}
by \cite[Cor.~VI.2.6]{ggk}, so
the result follows.
\end{proof}


In view of the link between Hausdorff dimension
error estimates and the tail (\ref{discardedtail}),
together with the bounding of terms in this tail by sums of products of approximation numbers provided by Lemma \ref{gohberglemma},
it will be important to establish upper bounds on the 
Taylor coefficients $\delta_n(s)$ for those $n$ where it is not computationally feasible to evaluate 
exactly
via periodic points.
We shall derive two distinct types of such upper bound, 
which we refer to as \emph{Euler bounds} and \emph{computed Taylor bounds}.
There is an Euler bound on $\delta_n(s)$ for each $n$, given as a simple closed form; 
this bound will be used for all sufficiently large values of $n$,
though for low values of $n$ may be too conservative for our purposes.
The finitely many
computed Taylor bounds
will be on the Taylor coefficients $\delta_{P+1}(s),\ldots, \delta_Q(s)$ where $P$ is the largest integer
for which we locate all period-$P$ points, and $Q$ is chosen so that the Euler bounds on $|\delta_n(s)|$ are
sufficiently sharp when $n>Q$.
 In view of Lemma \ref{gohberglemma}, the computed Taylor bounds will be derived by first bounding the finitely
 many approximation numbers $s_1(\l_{A,s}),\ldots, s_N(\l_{A,s})$, for some $N\in\N$, by explicitly computable quantities that we call \emph{computed approximation bounds}.
 The computations required to derive the computed approximation bounds are not onerous, the main task being the 
evaluation of numerical integrals defining certain $H^2$ norms
 (of the transfer operator images of a chosen orthonormal basis).

We shall approximate $\l_{A,s}$
by first projecting $H^2(D)$ onto the space of polynomials up to a given degree.
Let $\l_{A,s}:H^2(D)\to H^2(D)$ be a transfer operator,
where $D\subset \CC$ is an open disc of radius $\rho$ centred at $c$, and 
$\{m_k\}_{k=0}^\infty$ is the corresponding orthonormal basis of monomials, given by
\begin{equation}\label{monomial}
m_k(z)=\rho^{-k}(z-c)^k\,.
\end{equation}

\subsection{Approximation bounds}

\begin{defn}\label{alphadefn}
For $n\ge1$, define the $n^{th}$ \emph{approximation bound} $\alpha_n(s)$ to be
\begin{equation}\label{alphaexpression}
\alpha_n(s) = \left( \sum_{k=n-1}^\infty  \|\l_{A,s} (m_k)\|^2 \right)^{1/2} \,.
\end{equation}
\end{defn}


\begin{prop}\label{alphaapproximationbound}
For each $n\ge1$,
\begin{equation}\label{snalphanbound}
s_n(\l_{A,s}) \le \alpha_n(s)\,.
\end{equation}
\end{prop} 
\begin{proof}
For $f\in H^2(D)$ we can write
$$
f = \sum_{k=0}^\infty \hat f(k)\, m_k
$$
where the sequence $(\hat f(k))_{k=0}^\infty$ is square summable.
Define the rank-$(n-1)$ projection $\Pi_n:H^2(D)\to H^2(D)$ by
$$
\Pi_n(f) = \sum_{k=0}^{n-2} \hat f(k)\, m_k \,,
$$
where in particular $\Pi_1\equiv 0$.

The transfer operator $\l_{A,s}$ is approximated by the rank-$(n-1)$ operators $$\l_{A,s}^{(n)} := \l_{A,s} \Pi_n \,,$$
and  $\| \l_{A,s}- \l_{A,s}^{(n)} \|$ can be estimated
using the Cauchy-Schwarz inequality  as follows:
\begin{multline*}
\| (\l_{A,s} - \l_{A,s}^{(n)})f \|
= \| \sum_{k=n-1}^\infty \hat f(k)\, \l_{A,s}(m_k) \| 
 \le \sum_{k=n-1}^\infty | \hat f(k)| \| \l_{A,s}(m_k) \| \cr
 \le \left( \sum_{k=n-1}^\infty   \| \l_{A,s}(m_k) \|^2  \right)^{1/2}  \left( \sum_{k=n-1}^\infty |\hat f(k)|^2\right)^{1/2}  \cr
 \le \left( \sum_{k=n-1}^\infty   \| \l_{A,s}(m_k) \|^2  \right)^{1/2} \|f\| \,,\cr
\end{multline*}
and therefore 
$\| \l_{A,s}- \l_{A,s}^{(n)} \| \le \left( \sum_{k=n-1}^\infty   \| \l_{A,s}(m_k) \|^2  \right)^{1/2} = \alpha_n(s)$.
Since $\l_{A,s}^{(n)}$ has rank $n-1$,
it follows that $s_n(\l_{A,s}) \le \alpha_n(s)$, as required.
\end{proof}

\subsection{Contraction ratios}\label{contractionratiossubsection}

Let $C_i:H^2(D)\to H^2(D)$ be the \emph{composition operator}
$$
C_i f = f\circ T_i\,.
$$

The estimate arising in the following lemma motivates our definition below (see Definition \ref{hardydefinition})
of the \emph{contraction ratio} associated to a disc $D$ and subset $A\subset \N$.

\begin{lemma}\label{H2contractionratio}
Let $D$ and $D'$ be concentric discs, with radii $\rho$ and $\rho'$ respectively.
If, for $i\in A$, the image $T_i(D)$ is contained in $D'$, then for all $k\ge0$,
\begin{equation}
\| C_i(m_k)\| \le \left( \frac{\rho'}{\rho}\right)^k\,.
\end{equation}
\end{lemma}
\begin{proof}
Let $c$ denote the common centre of the discs $D, D'$.
If $z\in D$ then
$$|C_i(m_k)(z)| = \rho^{-k} |T_i(z)-c|^k < \rho^{-k} (\rho')^k = (\rho' /\rho)^k\,,$$
 so $\|C_i(m_k)\|\le  (\rho'/\rho)^k$, as required.
\end{proof}

For each $i\in A$, $s\in\R$, if the open disc $D$ is such that $-i\notin D$ then
define the \emph{weight function} $w_{i,s}:D\to\mathbb{C}$ by
$$
w_{i,s}(z) = \left(\frac{1}{z+i}\right)^{2s}\,, 
$$
and the \emph{multiplication operator} $W_{i,s}:H^2(D)\to H^2(D)$ by
$$
W_{i,s} f = w_{i,s} f\,.
$$
We may write
$$
\l_{A,s}  = \sum_{i\in A} W_{i,s} C_i\,,
$$
so that
$$
\| \l_{A,s}(m_k) \| \le 
\sum_{i\in A} \| W_{i,s} C_i(m_k) \|
\le
\sum_{i\in A} \| w_{i,s}\|_\infty  \|C_i(m_k) \|
 \,,
$$
and if $\rho_i'$ is such that $T_i(D)$ is contained in the concentric disc $D_i'$ of radius $\rho_i'$ then
Lemma \ref{H2contractionratio} implies that
\begin{equation}\label{rhoprimei}
\| \l_{A,s}(m_k) \| \le \sum_{i\in A} \| w_{i,s}\|_\infty (\rho_i'/\rho)^k\,.
\end{equation}

For our purposes it will be more convenient to work with a slightly simpler (and less sharp) version of
(\ref{rhoprimei}). This prompts the following definition:

\begin{defn}\label{hardydefinition}
Let $A\subset\N$ be finite, and $D\subset \mathbb{C}$ an open disc of radius $\rho$ such that
$\cup_{i\in A} T_i(D)\subset D$. Let $D'$ be the smallest disc, concentric with $D$, such that
$\cup_{i\in A} T_i(D)\subset D'$, and let
$\rho'$ denote the radius of $D'$.
The corresponding \emph{contraction ratio} $h=h_{A,D}$ is defined to be
\begin{equation}\label{hardyratiodefn}
h = h_{A,D} = \frac{\rho'}{\rho}\,. 
\end{equation}
\end{defn}

\begin{lemma}\label{hardyfirst}
Let $A\subset \N$ be finite, and $D$ an admissible disc, with contraction ratio $h=h_{A,D}$.
For all $k\ge0$,
\begin{equation}\label{normdecay}
\| \l_{A,s}(m_k) \| \le  h^k  \sum_{i\in A} \| w_{i,s}\|_\infty \,.
\end{equation}
\end{lemma}
\begin{proof}
If $D'$ is as in Definition \ref{hardydefinition} then $\rho'=\max_{i\in A} \rho_i'$ 
in the notation of (\ref{rhoprimei}), and the result follows from (\ref{rhoprimei}).
\end{proof}

\begin{cor}\label{ks}
Let $A\subset \N$ be finite, and $D$ an admissible disc, with contraction ratio $h=h_{A,D}$.
For all $n\ge1$,
\begin{equation}\label{salphabound}
s_n(\l_{A,s}) \le \alpha_n(s) \le  K_s  h^n
\end{equation}
where
\begin{equation}\label{ksdefn}
K_s =
\frac{  \sum_{i\in A} \| w_{i,s}\|_\infty }{ h  \sqrt{1 - h^2}}  \,.
\end{equation}
\end{cor}
\begin{proof}
Now
$$
\alpha_n(s) = \left( \sum_{k=n-1}^\infty  \|\l_{s} (m_k)\|^2 \right)^{1/2} 
$$
from Definition \ref{alphadefn} and Proposition \ref{alphaapproximationbound},
so Lemma \ref{hardyfirst} gives
$$
\alpha_n(s) \le   \left( \sum_{k=n-1}^\infty h^{2k} \right)^{1/2} \sum_{i\in A} \| w_{i,s}\|_\infty
= \frac{ h^{n-1}}{\sqrt{1 - h^2}}  \sum_{i\in A} \| w_{i,s}\|_\infty \,,
$$
and the result follows.
\end{proof}

\subsection{Euler bounds}

We can now derive the \emph{Euler bound} on the $n^{th}$ Taylor coefficient of the determinant:

\begin{prop}\label{eulerprop}
Let $A\subset \N$ be finite, and $D$ an admissible disc, with contraction ratio $h=h_{A,D}$.
If the transfer operator $\l_{A,s}$ has determinant $\det(I-z\l_{A,s})=1+\sum_{n=1}^\infty \delta_n(s)z^n$,
 then for all $n\ge1$,
 \begin{equation}\label{eulerequation}
|\delta_n(s)| \le \frac{K_s^n h^{n(n+1)/2}}{ \prod_{i=1}^n (1 - h^i)}\,.
\end{equation}
\end{prop}
\begin{proof}
By Lemma \ref{gohberglemma},
\begin{equation*}
|\delta_n(s)|
\leq
 \sum_{i_1 < \ldots < i_n}
\prod_{j=1}^n
s_{i_j}(\l_{A,s})
 \,,
\end{equation*}
so Corollary \ref{ks} gives
\begin{equation*}
|\delta_n(s)|
\leq
K_s^n
 \sum_{i_1 < \ldots < i_n} h^{i_1+\ldots+i_n} \,,
\end{equation*}
and the result follows by repeated geometric summation
(as first noted by Euler \cite[Ch.~16]{euler}).
\end{proof}

Henceforth we use the notation
\begin{equation}\label{eulerformula}
E_n(r) := \frac{r^{n(n+1)/2}}{\prod_{i=1}^n (1-r^i)} = \sum_{i_1< \ldots < i_n} r^{i_1+\ldots + i_n}\,,
\end{equation}
so that (\ref{eulerequation}) can be written as
\begin{equation}\label{eulerequation2}
|\delta_n(s)| \le K_s^n E_n(h)\,,
\end{equation}
and we define the righthand side of (\ref{eulerequation2}) 
(or equivalently of (\ref{eulerequation}))
to be the \emph{Euler bound} on the $n^{th}$ Taylor coefficient of the determinant.

\section{Computed approximation bounds}\label{cabsection}

For all $n\ge1$, the $n^{th}$ approximation bound 
$$\alpha_n(s) = \left( \sum_{k=n-1}^\infty  \|\l_{A,s} (m_k)\|^2 \right)^{1/2}$$
is, as noted in Proposition \ref{alphaapproximationbound}, an upper bound on the $n^{th}$ approximation number 
$s_n(\l_{A,s})$. 

Each $m_k$ is just a normalised monomial (\ref{monomial}), and the operator $\l_{A,s}$ is available in closed form,
so that 
$$\l_{A,s}(m_k)(z) = \sum_{i\in A} \frac{(T_i(z)-c)^k}{\rho^{k} (z+i)^{2s}}\,,$$
and we may use numerical integration 
to compute\footnote{Numerical integration capability
is available in computer packages such as Mathematica,
and these norms can be computed to arbitrary precision; although higher precision requires greater computing time,
these computations are relatively quick 
(e.g.~for the computations in \S \ref{E2section} these integrals were computed with 150 digit accuracy).}
 each Hardy norm $\|\l_{A,s} (m_k)\|$ as
\begin{equation}\label{hardynormexplicit}
\|\l_{A,s} (m_k)\|^2 = \int_0^1 \left| \sum_{i\in A}  \frac{(T_i(\gamma(t))-c)^k}{\rho^{k} (\gamma(t)+i)^{2s}}\right|^2 \, dt\,,
\end{equation}
where $\gamma(t)=c+\rho e^{2\pi it}$.

Evaluation of $\alpha_n(s)$ involves the tail sum
$\sum_{k=n-1}^\infty  \|\l_{A,s} (m_k)\|^2$,
and in practice we can bound this by the sum of an exactly computed long finite sum
$\sum_{k=n-1}^N  \|\l_{A,s} (m_k)\|^2$, for some $N \gg n$,
and a rigorous upper bound on 
$\sum_{k=N+1}^\infty  \|\l_{A,s} (m_k)\|^2$
using (\ref{normdecay}).
More precisely, we have the following definition: 

\begin{defn}
Given $n,N\in\N$, with $n\le N$, define the
\emph{lower and upper computed approximation bounds}, $\alpha_{n,N,-}(s)$ and $\alpha_{n,N,+}(s)$, respectively, by
\begin{equation}\label{alphanN}
\alpha_{n,N,-}(s) = \left( \sum_{k=n-1}^N \| \l_{A,s}(m_k)\|^2 \right)^{1/2}\,,
\end{equation}
and
\begin{equation}\label{alphanNplus}
\alpha_{n,N,+}(s) 
=
\left( \alpha_{n,N,-}(s)^2
+ \left(\sum_{i\in A} \|w_{i,s}\|_\infty\right)^2 \frac{h^{2(N+1)}}{1-h^2}\right)^{1/2}
\,.
\end{equation}
\end{defn}

Evidently the lower computed approximation bound $\alpha_{n,N,-}(s)$ is a lower bound for $\alpha_n(s)$, 
in view of the positivity of the summands in (\ref{alphaexpression}) and (\ref{alphanN}), while
Lemma \ref{upperlowerbounds} below establishes that the upper computed approximation bound
$\alpha_{n,N,+}(s)$ is an upper bound for $\alpha_n(s)$.
Moreover, both $\alpha_{n,N,+}(s)$ and $\alpha_{n,N,-}(s)$ are readily computable: they are given by finite sums and,
as already noted, the summands $\| \l_{A,s}(m_k)\|^2$ are computable to arbitrary precision.

\begin{lemma}\label{upperlowerbounds}
Let $s\in\R$. For all $n,N\in \N$, with $n\le N$,
\begin{equation}\label{alphaupperlower}
\alpha_{n,N,-}(s) \le \alpha_n(s) \le \alpha_{n,N,+}(s)\,.
\end{equation}
\end{lemma}
\begin{proof}
The inequality $\alpha_{n,N,-}(s) \le \alpha_n(s)$ is 
immediate from the definitions.
To prove that $\alpha_n(s) \le \alpha_{n,N,+}(s)$ note that
$$\alpha_n(s)^2 = \sum_{k=n-1}^N \|\l_{A,s}(m_k)\|^2 +  \sum_{k=N+1}^\infty \|\l_{A,s}(m_k)\|^2\,,$$
which together with (\ref{normdecay}) gives
$$\alpha_n(s)^2 \le  \sum_{k=n-1}^N \|\l_{A,s}(m_k)\|^2 + \left( \sum_{i\in A} \| w_{i,s}\|_\infty \right)^2 \frac{h^{2(N+1)}}{1-h^2}\,,
$$
and the result follows.
\end{proof}

\begin{remark}
The upper bound $\alpha_{n,N,+}(s)$ will be used in the sequel, as a tool
in providing rigorous estimates on Hausdorff dimension.
In practice $N$ will be chosen so that the values
$\alpha_{n,N,-}(s)$ and $\alpha_{n,N,+}(s)$ are close enough together that 
the inequality 
(\ref{alphaupperlower})
determines $\alpha_n(s)$
with precision far higher than that of the desired Hausdorff dimension estimate;
in particular, $N$ will be such that the difference 
$\alpha_{n,N,+}(s)-\alpha_{n,N,-}(s) = O(h^N)$ is
extremely small relative to the size of $\alpha_{n}(s)$.
\end{remark}

Combining (\ref{salphabound}) with (\ref{alphaupperlower}) immediately gives the exponential bound
\begin{equation}
\label{salphaboundN}
\alpha_{n,N,-}(s) \le K_s  h^n\quad\text{for all }n \le N\,,
\end{equation}
though the analogous bound for $\alpha_{n,N,+}(s)$ (which will be more useful to us in the sequel)
requires some extra care:

\begin{lemma}\label{nNboundlemma}
Let $s\in\R$. For all $n,N\in \N$, with $n\le N$,
\begin{equation}\label{nNbound}
\alpha_{n,N,+}(s) \le K_s(1 + h^{2(N+2-n)})^{1/2} h^n\,.
\end{equation}
\end{lemma}
\begin{proof}
Combining (\ref{salphaboundN}) with (\ref{alphanNplus})
gives 
$$
\alpha_{n,N,+}(s) 
\le
\left( (K_sh^n)^2
+ \left(\sum_{i\in A} \|w_{i,s}\|_\infty\right)^2 \frac{h^{2(N+1)}}{1-h^2}\right)^{1/2}\,,
$$
but (\ref{ksdefn}) gives
$$ \frac{ \left(\sum_{i\in A} \|w_{i,s}\|_\infty\right)^2}{1-h^2} = K_s^2h^2\,,$$
so
$$
\alpha_{n,N,+}(s) 
\le
\left( (K_sh^n)^2
+ K_s^2 h^{2(N+2)} \right)^{1/2} \,,
$$
and the result follows.
\end{proof}

The utility of (\ref{nNbound}) stems from the fact that in practice $N-n$ will be large, and that 
 for sufficiently small values of $n$ the following
more direct analogue of (\ref{salphaboundN}) can be used:

\begin{cor}\label{jqnscor}
Let $s\in\R$. 
Suppose $N,Q\in\N$, with $Q \le N$.
If 
\begin{equation}\label{jqns}
J = J_{Q,N,s} := K_s\left( 1+h^{2(N+2-Q)}\right)^{1/2}
\end{equation}
then
\begin{equation}\label{jqnshbound}
\alpha_{n,N,+}(s) \le J h^n\quad\text{for all }1\le n\le Q\,.
\end{equation}
\end{cor}
\begin{proof}
Immediate from Lemma \ref{nNboundlemma}.
\end{proof}

\begin{remark}
In practice $Q$ will be of some modest size, dictated by the computational resources at our disposal;
specifically, it will be chosen slightly larger than the largest $P\in \N$ for which it is feasible to compute all periodic points of period
$\le P$ (e.g.~in \S \ref{E2section}, when estimating the dimension of the set $E_2=E_{\{1,2\}}$,
we explicitly compute all periodic points up to period $P=25$, and in the proof of Theorem \ref{e2theorem} we choose $Q=28$).
The value $N$ will be chosen to be significantly larger than $Q$ (e.g.~in 
the proof of Theorem \ref{e2theorem}  we choose $N=600$).
Since $N+2-Q$ is large, $h^{N+2-Q}$ will be extremely small, and 
 $J=J_{Q,N,s}$ will be extremely close to $K_s$; ideally this closeness ensures that
 the two constants $J_{Q,N,s}$ and $K_s$ 
 are indistinguishable to the chosen level of working precision
 (e.g.~in the proof of Theorem \ref{e2theorem}, $N+2-Q=574$ and $h\approx 0.511284$,
 so $h^{N+2-Q}\approx 5.9 \times 10^{-168}$, whereas computations are performed to 150 decimal digit precision).
\end{remark}

\section{Computed Taylor bounds}\label{taylorsection}

In order to use the computed approximation bounds to provide a rigorous upper bound on the 
Taylor coefficients of the determinant $\det(I-z\l_{A,s})$,
we now fix a further natural number $M$, satisfying $M\le N$.
For any such $M$, it is convenient to define
the sequence $(\alpha_{n,N,+}^{M}(s) )_{n=1}^\infty$
to be the one whose $n^{th}$ term equals $\alpha_{n,N,+}(s)$ until $n=M$,
and whose subsequent terms are given
by the exponential upper bound on $s_n(\l_{A,s})$ and $\alpha_n(s)$ (cf.~(\ref{salphabound})):
\begin{equation}\label{alph}
\alpha_{n,N,+}^{M}(s) :=
\begin{cases}
\alpha_{n,N,+}(s) & \text{for }1\le n\le M\,, \cr
K_s h^n & \text{for }n>M\,.
\end{cases}
\end{equation}

This allows us to make the following definition:

\begin{defn}
Let $s\in\R$. For $n,M,N\in\N$ with $n\le M\le N$,
the \emph{Taylor bound} $\beta_{n,N,+}^{M}(s)$ is defined by
\begin{equation}\label{firstbeta}
\beta_{n,N,+}^{M}(s) := \sum_{i_1<\ldots< i_n} \prod_{j=1}^n \alpha_{i_j,N,+}^{M}(s) \,,
\end{equation}
where the sum is over those $\i=(i_1,\ldots,i_n)\in\N^n$
which satisfy $i_1< i_2 < \ldots< i_n$.
\end{defn}

As the name suggests, the Taylor bound $\beta_{n,N,+}^{M}(s)$ bounds the $n^{th}$
Taylor coefficient of the determinant $\det(I-z\l_{A,s})=1+\sum_{n=1}^\infty \delta_n(s)z^n$:

\begin{lemma}\label{betacbound}
Let $s\in\R$.
For $n, M, N\in \N$ with $n\le M\le N$,
\begin{equation}\label{cnalphatilde}
|\delta_n(s)|
\leq
 \beta_{n,N,+}^{M}(s)
 \,.
\end{equation}
\end{lemma}
\begin{proof}
Combining
(\ref{salphabound}), (\ref{alphaupperlower}) and (\ref{alph})
gives
\begin{equation*}\label{snalphanNplus}
s_n(\l_{A,s}) \le \alpha_{n,N,+}^{M}(s)\quad\text{for all }1\le n\le M \le N\,,
\end{equation*}
and combining this with Lemma \ref{gohberglemma} gives (\ref{cnalphatilde}).
\end{proof}

Note that $ \beta_{n,N,+}^{M}(s)$ is precisely the $n^{th}$ power series coefficient
for the infinite product $\prod_{i=1}^\infty (1+ \alpha_{i,N,+}^{M}(s) z)$, and
that the sum in (\ref{firstbeta}) is an infinite one; thus we will seek a computationally accessible approximation to 
$ \beta_{n,N,+}^{M}(s)$.
We expect that $ \beta_{n,N,+}^{M}(s)$ is well approximated by the $n^{th}$ power series coefficient
for the \emph{finite} product $\prod_{i=1}^M (1+ \alpha_{i,N,+}^{M}(s) z) = \prod_{i=1}^M (1+ \alpha_{i,N,+}(s) z)$,
namely the value $\beta_{n,N,+}^{M,-}(s)$ 
defined as follows:

\begin{defn}
Let $s\in\R$. For $n,M,N\in\N$ with $n\le M\le N$,
the \emph{lower computed Taylor bound} $\beta_{n,N,+}^{M,-}(s)$ is defined as
\begin{equation}\label{betanNplusMminusalternative}
\beta_{n,N,+}^{M,-}(s)
:=
  \sum_{i_1<\ldots< i_n\le M} \prod_{j=1}^n \alpha_{i_j,N,+}(s)\,.
 \end{equation}
\end{defn}

\begin{remark}
\item[\, (i)]
The fact that $\beta_{n,N,+}^{M,-}(s)$ is defined in terms of upper computed approximation bounds 
$\alpha_{i_j,N,+}(s)$,
together with the finiteness of the sum (and product) in (\ref{betanNplusMminusalternative}), 
ensures that $\beta_{n,N,+}^{M,-}(s)$ can be computed (to arbitrary precision).
\item[\, (ii)]
Clearly, an equivalent definition of 
$\beta_{n,N,+}^{M,-}(s)$ is
\begin{equation}\label{betanNplusMminus}
\beta_{n,N,+}^{M,-}(s)
=
 \sum_{i_1<\ldots< i_n \le M} \prod_{j=1}^n \alpha_{i_j,N,+}^{M}(s)\,.
\end{equation}
\end{remark}


The lower computed Taylor bound
$\beta_{n,N,+}^{M,-}(s)$ is obviously smaller than
the Taylor bound
 $\beta_{n,N,+}^{M}(s)$, though in view of (\ref{cnalphatilde})
 we require an \emph{upper computed Taylor bound} (introduced in Definition \ref{uppercomputedtaylorbound} below)
 that is larger than $\beta_{n,N,+}^{M}(s)$.
The following result estimates
the difference 
$\beta_{n,N,+}^{M}(s) - \beta_{n,N,+}^{M,-}(s)$, and subsequently
(see Definition \ref{uppercomputedtaylorbound})
 provides the inspiration for the definition of the
upper computed Taylor bound:

\begin{lemma}\label{betadifference}
Let $s\in\R$.
Given $Q,M,N\in\N$ with $Q\le M\le N$,
 and $J=J_{Q,N,s}$ defined by (\ref{jqns}),
\begin{equation}\label{betadifferencerequiredbound}
\beta_{n,N,+}^{M}(s) - \beta_{n,N,+}^{M,-}(s)
\le \sum_{l=0}^{n-1} J^{n-l} \beta_{l,N,+}^{M,-}(s)\, h^{M(n-l)} E_{n-l}(h)\quad\text{for all }1\le n\le Q \,.
\end{equation}
\end{lemma}
\begin{proof}
Let $n$ be such that $1\le n\le Q$.
The set
$
\I_n := \{ \i = (i_1,\ldots,i_n) \in \N^n: i_1 <\ldots < i_n\} 
$
can be partitioned as 
$\I_n = \bigcup_{l=0}^n \I_n^{(l)}$,
where the $\I_n^{(l)}$ are defined by
\begin{equation*}
\I_n^{(l)} =
\begin{cases}
\{ \i = (i_1,\ldots,i_n)\in\I_n:  M < i_{1}\} & \text{if }l=0 \,, \cr
\{ \i = (i_1,\ldots,i_n)\in\I_n: i_l \le M < i_{l+1}\} & \text{if } 1\le l \le n-1\,, \cr
\{ \i = (i_1,\ldots,i_n)\in\I_n: i_n \le  M \} & \text{if }l=n\,.
\end{cases}
\end{equation*}
Define
$$
\beta_{n,N,+}^{M, (l)}(s) 
:= \sum_{\i \in \I_n^{(l)}} 
\prod_{j=1}^n \alpha_{i_j,N,+}^{M}(s) 
\quad\text{for each }0\le l\le n\,,
$$
so that in particular
\begin{equation}\label{inparticularnminus}
\beta_{n,N,+}^{M, (n)}(s) = \beta_{n,N,+}^{M, -}(s) \,.
\end{equation}

With this notation, and since $\I_n = \bigcup_{l=0}^n \I_n^{(l)}$,
 we can express $\beta_{n,N,+}^{M}(s)$ as
\begin{equation}\label{canexpress}
\beta_{n,N,+}^{M}(s) 
= \sum_{\i \in \I_n} \prod_{j=1}^n \alpha_{i_j,N,+}^{M}(s)   
= \sum_{l=0}^{n} \beta_{n,N,+}^{M, (l)}(s)  \,.
\end{equation}
Combining (\ref{inparticularnminus}) and (\ref{canexpress}) gives
\begin{equation}\label{epsilonsum}
\beta_{n,N,+}^{M}(s) - \beta_{n,N,+}^{M,-}(s)
= \sum_{l=0}^{n-1} \beta_{n,N,+}^{M, (l)}(s) \,.
\end{equation}

In order to bound each $\beta_{n,N,+}^{M, (l)}(s)$ in (\ref{epsilonsum}) we
use the fact that $\alpha_{i,N,+}^{M}(s) \le J h^i$ for all $1\le i\le Q$ (see Corollary \ref{jqnscor}) to obtain
\begin{equation}\label{manipulation}
\beta_{n,N,+}^{M, (l)}(s) 
= \sum_{\i \in \I_n^{(l)}} \prod_{j=1}^n \alpha_{i_j,N,+}^{M}(s)
\le J^{n-l} \sum_{\i \in \I_n^{(l)}}  h^{i_{l+1} +\ldots + i_n} \prod_{j=1}^l \alpha_{i_j,N,+}^{M}(s) \,,
\end{equation}
and introducing $\j = (\iota_1,\ldots,\iota_{n-l})\in \I_{n-l}$ with $i_{l+k} = \iota_k +M$ for $1\le k\le n-l$, we can re-express the righthand side of (\ref{manipulation}) to obtain
\begin{equation*}
\beta_{n,N,+}^{M, (l)}(s)
\le J^{n-l} \left( \sum_{\i \in \I_l^{(l)}}  \prod_{j=1}^l \alpha_{i_j,N,+}^{M}(s) \right)
\left(  \sum_{\j\in \I_{n-l}} h^{(n-l)M} h^{\iota_1+\ldots+\iota_{n-l} }  \right) \,,
\end{equation*}
and therefore
\begin{equation}\label{betanlM}
\beta_{n,N,+}^{M, (l)}(s) 
\le J^{n-l} \beta_{l,N,+}^{M,-}(s)\,  h^{M(n-l)} \ E_{n-l}(h)\,.
\end{equation}

Now combining (\ref{epsilonsum}) and (\ref{betanlM}) gives the required bound (\ref{betadifferencerequiredbound}).
\end{proof}

\begin{remark}
In practice the $l=n-1$ term on the righthand side of
(\ref{betadifferencerequiredbound}) tends to be the dominant one, as $M$ is chosen large enough so that
$h^M$ is extremely small.
\end{remark}
 
\begin{defn}\label{uppercomputedtaylorbound} 
Let $s\in\R$. For $n, Q,M,N\in\N$ with $n\le Q\le M\le N$,
define the \emph{upper computed Taylor bound} $ \beta_{n,N,+}^{M,+}(s)$ by
$$
 \beta_{n,N,+}^{M,+}(s)
 :=
 \beta_{n,N,+}^{M,-}(s) +  \sum_{l=0}^{n-1} J_{Q,N,s}^{n-l}\, \beta_{l,N,+}^{M,-}(s)\, h^{M(n-l)} E_{n-l}(h) \,.
 $$
 \end{defn}
 
 From
 Lemma \ref{betadifference} it then follows that the upper computed Taylor bound
 $\beta_{n,N,+}^{M,+}(s)$
  is indeed larger than the Taylor bound
 $\beta_{n,N,+}^{M}(s)$:

 \begin{cor}\label{betadifferencecor}
Let $s\in\R$.
If $Q,M,N\in\N$ with $Q\le M\le N$,
then 
$$
\beta_{n,N,+}^{M}(s) \le \beta_{n,N,+}^{M,+}(s)\quad
\text{for all }1\le n\le Q \,.
$$
\end{cor}
\begin{proof}
Immediate from Lemma \ref{betadifference} and Definition \ref{uppercomputedtaylorbound}.
\end{proof}

 Finally, we deduce that the $n^{th}$ Taylor coefficient $\delta_n(s)$ of the determinant $\det(I-z\l_{A,s})$
 can be bounded in modulus by the upper computed Taylor bound
 $\beta_{n,N,+}^{M,+}(s)$ (a quantity we can compute to arbitrary precision):

\begin{prop}\label{taylorboundprop}
Let $s\in\R$.
If $Q,M,N\in\N$ with $Q\le M\le N$,
then 
$$
|\delta_n(s)| \le \beta_{n,N,+}^{M,+}(s)\quad\text{for all } 1\le n\le Q\,.
$$
\end{prop}
\begin{proof}

Lemma \ref{betacbound} gives
$ |\delta_n(s)| \le \beta_{n,N,+}^{M}(s)$, and
Corollary \ref{betadifferencecor} gives
$\beta_{n,N,+}^{M}(s) \le \beta_{n,N,+}^{M,+}(s)$,
so the result follows.
\end{proof}

\begin{remark}
In \S \ref{E2section}, 
for the computations in the proof of Theorem \ref{e2theorem},
we choose $N=600$, $M=400$, and $Q=28$, using
Proposition \ref{taylorboundprop} to obtain the upper bound on $|\delta_n(s)|$ for $P+1=26\le n\le 28$,
having explicitly evaluated $\delta_n(s)$ for $1\le n\le 25$ using periodic points of period up to $P= 25$.
\end{remark}

\section{The Hausdorff dimension of $E_2$}\label{E2section}

Here we consider the set $E_2$, corresponding to the choice $A=\{1,2\}$.
We shall suppress the set $A$ from our notation, writing $\l_s$ instead of $\l_{A,s}$.

The approximation $s_N$ to $\text{dim}(E_2)$, based on periodic points of period up to $N$,
is the zero (in the interval $(0,1)$) of the function $\D_N$ defined by (\ref{DNdefn});
these approximations are tabulated in Table 1 for $18\le n\le 25$. 
We note that the 24th and 25th approximations to $\text{dim}(E_2)$ share
the first 129 decimal digits
\begin{align*}
0.&5312805062772051416244686473684717854930591090183987798883978039 \\
& 27529535643831345918109570181185239880428057243075187633422389339 \,
\end{align*}
though the rate of convergence gives confidence that the first 139 digits
\begin{align*}
0.&531280506277205141624468647368471785493059109018398779888397803927529 \\
& 5356438313459181095701811852398804280572430751876334223893394808223090 \,
\end{align*}
of $s_{25}$ are in fact correct digits of $\text{dim}(E_2)$.

It turns out that we can \emph{rigorously} 
justify around three quarters of these decimal digits,
proving that the first 100 digits are correct.
In fact we prove slightly more than that, 
by setting $s^-$ to be the value
\begin{align*}
s^- =
0.&531280506277205141624468647368471785493059109018398 \\
& 77988839780392752953564383134591810957018118523987\,,
\end{align*}
and setting $s^+ = s^- + 2/10^{101}$ to be the value
\begin{align*}
s^+ =
0.&531280506277205141624468647368471785493059109018398 \\
& 77988839780392752953564383134591810957018118523989\,.
\end{align*}

\begin{table}
\begin{equation*}
\begin{tabular}{|r|r|}
\hline
$n$  & $s_n$ \qquad\qquad\qquad\qquad\quad\qquad\qquad\qquad\qquad\qquad\qquad \\
\hline 
$18$ & 0.531280506277205141624468647368471785493059109018398779888397803927529535645 \cr
$$ & 596972005085668529391352118806494054592120629038239974478243258576620540205 \cr
$19$ & 0.531280506277205141624468647368471785493059109018398779888397803927529535643 \cr
$$ & 831345931151408384198942403518425963034455124305471103063941900681921725781 \cr
$20$ & 0.531280506277205141624468647368471785493059109018398779888397803927529535643 \cr
$$ & 831345918109570144457186603287266737112934351614056377793361034907544181115 \cr
$21$ & 0.531280506277205141624468647368471785493059109018398779888397803927529535643 \cr
$$ & 831345918109570181185239840988322512589524907498366765561230541095944497891 \cr
$22$ & 0.531280506277205141624468647368471785493059109018398779888397803927529535643\cr
$$ & 831345918109570181185239880428057259226147992212780800516214656456345194120 \cr
$23$ & 0.531280506277205141624468647368471785493059109018398779888397803927529535643 \cr
$$ &    831345918109570181185239880428057243075187635944921448427780108909724612227 \cr
$24$ & 0.531280506277205141624468647368471785493059109018398779888397803927529535643 \cr
$$ & 831345918109570181185239880428057243075187633422389339330546198723829886067 \cr
$25$ & 0.531280506277205141624468647368471785493059109018398779888397803927529535643 \cr
$$ & 831345918109570181185239880428057243075187633422389339480822309014454563836 \cr
\hline
\end{tabular}
\end{equation*}
\caption{Approximations $s_n\approx \text{dim}(E_2)$; each $s_n$ is a zero of a truncation $\D_n$ (formed using only periodic points of period $\le n$) of the function $\D$}
\end{table}

We then claim:

\begin{theorem}\label{e2theorem}
The Hausdorff dimension of $E_2$ lies in the interval $(s^-,s^+)$.
\end{theorem}
\begin{proof}
We will show that $\D(s^-)$ and $\D(s^+)$ take opposite signs, 
and deduce that $\dim(E_A)$, as the zero of $\D$, lies between $s^-$ and $s^+$.


Let $D\subset\mathbb{C}$ be the open disc centred at $c$, of radius $\rho$,
where $c$ is the largest real root of the polynomial
$$
128c^7 + 768c^6 + 1296c^5  - 192c^4 - 1764c^3 - 108c^2 + 819c -216   \,,
$$
so that
$$c\approx
0.758687144013554292899790137015621955739402945444266741967051997691009 \,,
$$
and
\begin{equation}\label{crho}
\rho=
\frac{-c+\sqrt{-6c+5c^2+12c^3+4c^4}}{2c}\,,
\end{equation}
so that
$$
\rho\approx
0.957589818521375342814351002388265920293251603461349541441037951859499 \,.
$$
The relation (\ref{crho}) ensures that $T_1(c-\rho)$ and $T_2(c+\rho)$ are equidistant from $c$, and this common distance is denoted by $\rho' = T_1(c-\rho) - c = c - T_2(c+\rho)$, so that
\begin{equation*}\label{rhoprime}
\rho' \approx 
0.48960063348666271539624547964205669003751747416510762619582637319401 \,.
\end{equation*}
The specific choice of $c$ is to ensure that the 
contraction ratio $h = \rho'/\rho$ is minimised, taking the value
\begin{equation*}
h = \frac{\rho'}{\rho} \approx
0.51128429314616176482942956363790038479511374855036304746799036536341 \,.
\end{equation*}

\begin{figure}[!h]
\begin{center}
 \includegraphics[]{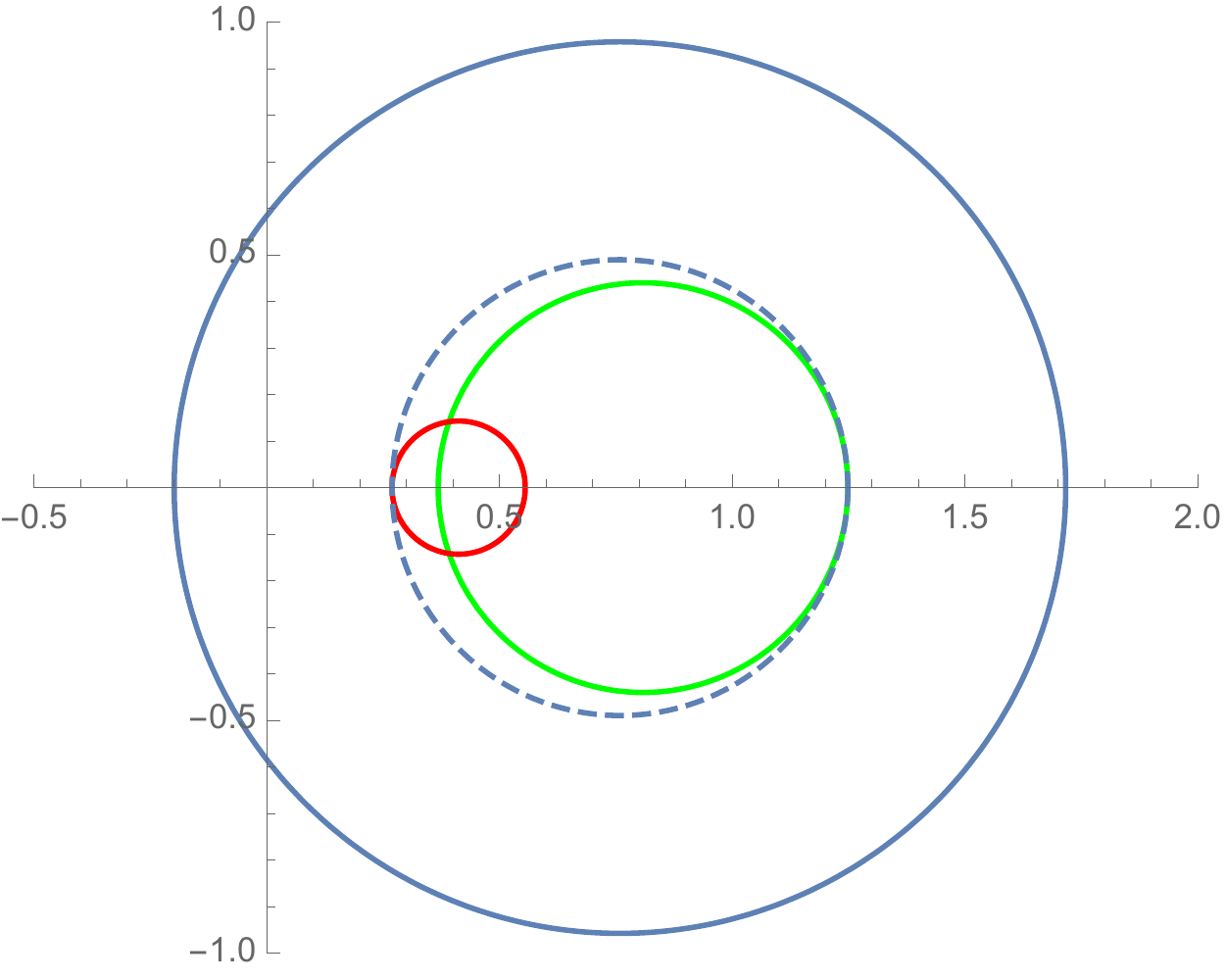}
\caption{Inner disc $D'$ (dashed) contains images $T_1(D)$, $T_2(D)$ of the outer disc $D$, in the rigorous bound on the dimension of $E_2$}
\end{center}
\end{figure}

Having computed the points of period up to $P= 25$ we can 
form the functions
$s\mapsto \delta_n(s)$ for $1\le n\le 25$,
and evaluate these at $s=s^-$
(cf.~Table 2)  to give
\begin{equation}\label{sminus22}
\D_{25}(s^-) = 1+\sum_{n=1}^{25} \delta_n(s^-) =
(-1.584605810787991617286291643870\ldots)
\times 10^{-101}
< 0\,,
\end{equation}
and at $s=s^+$ to give
\begin{equation}\label{splus22}
\D_{25}(s^+) =
1+\sum_{n=1}^{25} \delta_n(s^+) =
(1.454514082498475271478438451769\ldots)
 \times 10^{-101} >0\,.
\end{equation}

We now aim to show that the  approximation $\D_{25}$ is close enough to 
$\D$ for (\ref{sminus22}) and (\ref{splus22}) to imply, respectively, the negativity of $\D(s^-)$
and the positivity of $\D(s^+)$.
In other words, we seek to bound the tail $\sum_{n=26}^\infty \delta_n(s)$, and this will
be achieved by bounding the individual
Taylor coefficients $\delta_n(s)$, for $n\ge 26=P+1$.
It will turn out that for $n\ge29$ the cruder Euler bound on $\delta_n(s)$ is sufficient,
while for $26\le n\le 28$ we will use the Taylor bounds described in \S \ref{taylorsection}.
More precisely, for $P+1=26\le n\le 28=Q$ we will use the upper computed Taylor bound\footnote{As will be noted shortly,
the upper computed Taylor bound we use agrees with the corresponding Taylor bound to over 200 decimal digits,
so in particular the two quantities are indistinguishable at the 150 digit precision level of these computations.}
$\beta_{n,N,+}^{M,+}(s)$ for suitable $M,N\in\N$.

Henceforth let $Q=28$, $M=400$, $N=600$ (so that in particular $Q\le M\le N$, as was assumed throughout \S \ref{taylorsection})
and consider the case $s=s^-$.

We first evaluate\footnote{As described in \S \ref{cabsection}, 
(\ref{hardynormexplicit}) can be readily evaluated to
arbitrary precision using numerical integration; for this particular computation the precision level used was 150 decimal places.}  
the $H^2(D)$ norms of the monomial images $\l_{s} (m_k)$ for $0\le k\le N=600$.
These norms are decreasing in $k$;
Table 3 contains the first few evaluations, for $0\le k\le 10$,
while for $k= 600$ we have
$$
\| \l_{s}(m_{600}) \|= 
(2.297607298251023508986187604945746\ldots) \times 10^{-176} \,.
$$

Using these norms $\|\l_s(m_k)\|$  we then evaluate, for $1\le n\le M=400$,  the
upper computed approximation bounds
$\alpha_{n,N,+}(s) = \alpha_{n,600,+}(s)$ defined (cf.~(\ref{alphanNplus})) by\footnote{Note that 
$h\approx 0.511284$ and $N=600$,
so $\frac{h^{2(N+1)}}{1-h^2} \le 8.8 \times 10^{-351}$.
Moreover (\ref{weightboundsrange3}) gives
$\sum_{i=1}^2 \|w_{i,s}\|_\infty \le 1.81$, thus
$ (\sum_{i=1}^2 \|w_{i,s}\|_\infty)^2 \frac{h^{2(N+1)}}{1-h^2} \le 2.9 \times 10^{-350}$.
Combining these bounds with the values taken by $\alpha_{n,N,+}(s) $, 
it follows that for $1\le n\le 400$, 
the approximation bound $\alpha_n(s) = ( \sum_{k=n-1}^\infty  \|\l_{s} (m_k)\|^2 )^{1/2}$ agrees with both computed approximation bounds
$\alpha_{n,N,-}(s)$ and $\alpha_{n,N,+}(s)$ to at least 200 decimal places, a level well beyond the
desired precision used in these calculations.} 
$$
\alpha_{n,N,+}(s) 
= \left( \sum_{k=n-1}^N \| \l_{s}(m_k)\|^2 
+ \left(\sum_{i=1}^2 \|w_{i,s}\|_\infty\right)^2 \frac{h^{2(N+1)}}{1-h^2}\right)^{1/2} \,.
$$
These bounds are decreasing in $n$; Table 4 contains the first few evaluations, for $1\le n\le 10$,
while for $n=400$ we have
$$
\alpha_{400,600,+}(s) = 
(3.806826780744825698066314723072781\ldots) \times 10^{-147} \,.
$$

The upper computed approximation bounds $\alpha_{n,600,+}(s)$ are then used to form the 
upper computed Taylor bounds\footnote{The difference
$ \beta_{n,N,+}^{M,+}(s) -  \beta_{n,N,+}^{M,-}(s)
=
\sum_{l=0}^{n-1} J_{Q,N,s}^{n-l}\, \beta_{l,N,+}^{M,-}(s)\, h^{M(n-l)} E_{n-l}(h)$ is smaller than
$1.86 \times 10^{-210}$ for $26\le n\le 28=Q$, so in fact the upper and lower computed Taylor bounds,
and the Taylor bound $\beta_{n,N,+}^M(s)$, agree to well beyond the 150 decimal place precision used in these computations.}
 $\beta_{n,N,+}^{M,+}(s)
=  \beta_{n,N,+}^{M,-}(s) +  \sum_{l=0}^{n-1} J_{Q,N,s}^{n-l}\, \beta_{l,N,+}^{M,-}(s)\, h^{M(n-l)} E_{n-l}(h)$,
where
$$
\beta_{n,N,+}^{M,-}(s) = \beta_{n,600,+}^{400,-}(s) =  \sum_{i_1<\ldots< i_n\le 400} \prod_{j=1}^n \alpha_{i_j,600,+}(s)\,,
$$
which for $26\le n\le 28=Q$ are\footnote{See also Table 6
for computations of $\beta_{n,N,+}^{M,+}(s)$ for $1\le n\le 28=Q$.}
$$
\beta_{26,N,+}^{M,+}(s)
=
(7.0935010683530957339350457686786431427508\ldots) \times 10^{-103},
$$
$$
\beta_{27,N,+}^{M,+}(s)
=
(7.0379118021870691622913562125699156503586\ldots) \times 10^{-111},
$$
$$
\beta_{28,N,+}^{M,+}(s)
= 
(3.5360715444914082167026977943200738452867\ldots) \times 10^{-119},
$$
so in particular Proposition \ref{taylorboundprop} gives
\begin{equation}\label{intermediateE2}
\sum_{n=26}^{28} |\delta_n(s)|
\le
\sum_{n=26}^{28} \beta_{n,N,+}^{M,+}(s)
< 7.1
\times 10^{-103} .
\end{equation}

It remains to derive the Euler bounds on the Taylor coefficients $\delta_n(s)$ for $n\ge 29$.
For $s>0$, the functions $w_{1,s}(z) = 1/(z+1)^{2s}$ and $w_{2,s}(z)=1/(z+2)^{2s}$ have maximum modulus on $D$ when $z=c-\varrho$, so 
\begin{equation}\label{weightbound}
\|w_{1,s}\|_\infty = 1/(1+c-\varrho)^{2s}\quad \text{and} \quad \|w_{2,s}\|_\infty = 1/(2+c-\varrho)^{2s}\,.
\end{equation}

A computation using (\ref{weightbound}) gives
\begin{equation}\label{weightboundsrange1}
\|w_{1,s}\|_\infty 
\le 
1.2657276413750668025007241047661655434034644495987711959332997
\end{equation}
and
\begin{equation}\label{weightboundsrange2}
\|w_{2,s}\|_\infty 
\le  
0.5351507690357290789991731014616306223833750046974228167583536
\,,
\end{equation}
thus
\begin{equation}\label{weightboundsrange3}
\|w_{1,s}\|_\infty + \|w_{2,s}\|_\infty 
\le 
1.8008784104107958814998972062277961657868394542961940127
\,,
\end{equation}
and therefore $K_s= (\|w_{1,s}\|_\infty + \|w_{2,s}\|_{\infty})/( h \sqrt{1-h^2})$ 
is bounded by
\begin{equation}
K_s 
\le  
4.098460062897625162727128104751085223751087056801141844 \,.
\end{equation}

Now $|\delta_n(s)| \le K_s^n E_n(h)$, and we readily compute (see also Table 5) that
$$
K_s^{29} E_{29}(h) 
< 3.991837779947559
\times 10^{-109}\,,
$$
$$
K_s^{30} E_{30}(h) 
< 2.976234382308237
\times 10^{-117}\,,
$$
and we easily bound
\begin{equation}\label{easybound2}
\left| \sum_{n=29}^\infty \delta_n(s) \right| \le
\sum_{n=29}^\infty K_s^n E_n(h) <  4 \times 10^{-109}\,.
\end{equation}

Combining (\ref{easybound2})
with (\ref{intermediateE2}) gives, for $s = s^-$,
\begin{equation}\label{easyboundcombined}
\left| \sum_{n=26}^\infty \delta_n(s) \right| 
< 7.2 \times 10^{-103}\,.
\end{equation}

Combining (\ref{easyboundcombined}) with
(\ref{sminus22}) 
then gives
\begin{equation}
\D(s^-) = 1+\sum_{n=1}^\infty \delta_n(s^-) < 0 \,.
\end{equation}

It remains to show that $\D(s^+)$ is positive.
In view of (\ref{splus22}), 
for this it is sufficient to show that $\left| \sum_{n=26}^\infty \delta_n(s) \right| < 10^{-101}$
for $s=s^+$.
In fact the stronger inequality (\ref{easyboundcombined}) (which we have proved for $s=s^-$) can also be established
for $s=s^+$, using the same general method as for $s=s^-$,
since the intermediate computed values for the
norms $\|\l_s(m_k)\|$, computed approximation bounds $\alpha_{n,N,+}(s)$, computed Taylor bounds
$\beta_{n,N,+}^{M,+}(s)$,
and Euler bounds $K_s^nE_n(h)$,
are sufficiently close to those for $s=s^- = s^+ - 2/10^{101}$. 
Combining (\ref{splus22}) with inequality (\ref{easyboundcombined}) for $s=s^+$ gives
the required positivity
\begin{equation}
\D(s^+) = 1+\sum_{n=1}^\infty \delta_n(s^+) > 0\,.
\end{equation}

The map $s\mapsto \D(s)$ is continuous and increasing, so the fact that $\D(s^-) < 0 < \D(s^+)$
implies that its unique zero (which is equal to the dimension) is contained in $(s^-,s^+)$.
\end{proof}

\begin{remark}
If, as in Theorem \ref{e2theorem}, our aim is to rigorously justify 100 decimal places of the computed approximation $s_P$
to the Hausdorff dimension, then roughly speaking $P$ should be chosen so that the modulus of the
tail $\sum_{n=P+1}^\infty \delta_n(s)$ can be shown to be somewhat smaller than $10^{-100}$ for $s\approx s_P$. Since $|\delta_n(s)|$ is bounded above by the upper computed Taylor bound $\beta_{n,N,+}^{M,+}(s)$, the fact that
$\beta_{26,N,+}^{M,+}(s) < 7.1 \times 10^{-103}$ (see Table 6) for suitably large $M, N$, together with the rapid decay 
(as a function of $n$) of these bounds, suggests that we may choose $P=25$, i.e.~that it suffices to explicitly locate the periodic points of period $\le 25$.

The choice of the value $Q$ is relatively unimportant, as the upper computed Taylor bounds are only slightly more time consuming to compute than the (instantaneously computable) Euler bounds; in the proof of Theorem  \ref{e2theorem} we chose $Q$ such that the Euler bounds $K_s^nE_n(h)$ were substantially smaller than $10^{-100}$ for $n>Q$ (our choice $Q=28$ has this property, as does any larger $Q$, and indeed the choice $Q=27$ may also be feasible, cf.~Table 5).

The values $M$ and $N$ 
are chosen large enough to ensure that the bound (\ref{Taylorapproximationbound}) on $|\delta_n(s)|$ is rendered essentially as sharp as possible using our method (see Proposition \ref{alphaapproximationbound}) of bounding approximation numbers by approximation bounds; equally, the values $M$ and $N$ are of course chosen small enough to allow the $\beta_{n,N,+}^{M,+}(s)$ to be evaluated in reasonable time.
\end{remark}

\newpage

\begin{table}
\begin{equation*}
\begin{tabular}{|r|r|}
\hline
$n$  & $\delta_n(s)$ \qquad\qquad \qquad \qquad\qquad\qquad\qquad\qquad\qquad\qquad \\
\hline 
$0$ & $1.00000000000000000000000000000000000000000000000000000000000000000$ \cr
$1$ & $-0.76853713973783664059555880616494947204728086574720496608180647371$ \cr
$2$ & $-0.26021976366093635437716029700462536967772836185911363417403187305$ \cr
$3$ & $0.02765000991360418692432023659242068499500195127534488833324814033$ \cr
$4$ & $0.00112374639478016294259719123593672797015144554465624446191255585$ \cr
$5$ & $-0.0000167586893262829053963800867331298214048355450207764533921201$ \cr
$6$ & $-9.4420708961650542507455077292536505589082088322403413391248\times 10^{-8}$ \cr
$7$ & $2.002631154264594909155061001947400470978400350528119075400\times 10^{-10}$ \cr 
$8$ & $1.608231764929372179126081732895703844557686618425008678027 \times 10^{-13}$ \cr
$9$ &  $-4.893556025044534292717368610780157833682056735012684894922\times 10^{-17}$\cr
$10$ & $-5.651862783135703772626682291328447783083215443130743201938\times 10^{-21}$\cr
$11$ & $2.479739513988220884083251961840541108885795134827792410316\times 10^{-25}$\cr
$12$ & $4.136401121594147971038636851634368411897631129671577551043\times 10^{-30}$\cr
$13$ & $-2.624756973891155869061345045877184074663200387233753736326 \times 10^{-35}$\cr
$14$ & $-6.338941892590978104708773275720546500369680606966992807613\times 10^{-41}$\cr
$15$ & $5.828730628244270965574851653454290059269126664657684771853 \times 10^{-47}$\cr
$16$ & $2.041279162245973098261089683937621906304968188235420213343\times 10^{-53}$\cr
$17$ & $-2.723457305394335826243564087510129051800792696839009712149\times 10^{-60}$\cr
$18$ & $-1.384617032922521104261197591114142447361756695512763047462 \times 10^{-67}$\cr
$19$ & $2.682974662699446094806576747549474738085235119849542148518 \times 10^{-75}$\cr
$20$ & $1.981785501971166402977117745705012463041957402989929807212 \times 10^{-83}$\cr
$21$ & $-5.581047861819085366787152065083481128824923252068053906083 \times 10^{-92}$\cr
$22$ & $-5.99310412272224270828369069621010481279832938275818217131 \times 10^{-101}$\cr
$23$ & $2.45423524572073669786403014748119272764064394193220008396 \times 10^{-110}$\cr
$24$ & $3.83313875710563588641117264949062942911961150094959790393 \times 10^{-120}$\cr
$25$ & $-2.28353558134974299687217160340929697313195978714612008350 \times 10^{-130}$\cr
\hline
\end{tabular}
\end{equation*}
\caption{Exact (to the given precision) Taylor coefficients $\delta_n(s)$ for the determinant 
$\det(I-z\l_s)=1+\sum_{n=1}^\infty \delta_n(s) z^n$ for $E_2$ transfer operator $\l_s$ with $s=s^-$}
\end{table}

\begin{table}[H]
\begin{equation*}
\begin{tabular}{|r|r|}
\hline
$k$  & $\|\l_{s}(m_k)\|$ \qquad\qquad \qquad\qquad\qquad\qquad\qquad\qquad\qquad \\
\hline  
$0$ & 1.0270790783376427840070677716704413443556765790531396305598028764891 \cr
$1$ &  0.3937848239109563523505359783093188356154137707117445532439663747781  \cr
$2$ &  0.1714591180108060752265529053281347472947978460219396035391070667691\cr
$3$ &  0.0784792797693053045975192814445601433860119013766718128894674834037 \cr
$4$ &   0.0368985150737907248938351875080596507139356576758391651885254166051 \cr
$5$ &   0.0176517923866933707140642945427091399723431868286590018130953901715 \cr
$6$ &    0.0085477463829669713632455215487177327086334690252589671713112735110\cr
$7$ &   0.0041762395195693491669377402131475622078401074275749884365926135321 \cr
$8$ &   0.0020541561464629266556123666395075007822413063382433235450055746854 \cr
$9$ &    0.0010155981305058227350650668511905652569101368771929481102954501965\cr
$10$& 0.0005041555520431887383182315523421205104649185947907910778866174462\cr
\hline
\end{tabular}
\end{equation*}
\caption{$H^2(D)$ norms $\|\l_{s}(m_k)\|$ for $E_2$ transfer operator $\l_{s}$ with $s=s^-$, and
disc $D$
centred at $c\approx 0.758687$, of radius $\rho\approx 0.957589$}
\end{table}

\vskip 2cm

\begin{table}[H]
\begin{equation*}
\begin{tabular}{|r|r|}
\hline
$n$  & $\alpha_{n,N,+}(s)$ \qquad \qquad\qquad\qquad\qquad\qquad\qquad\qquad \\
\hline 
$1$ &  1.1168188427493689387528468664326403365355467885350235197794937054219 \cr
$2$ &  0.4386261441833328551532057324432712062653963332311641747430735557039 \cr
$3$ &  0.1932004317245564565674981131652477003552483794394786484356380783895 \cr
$4$ &  0.0890403148551906045843926762042532922519090868320365095369073804490\cr
$5$ &  0.0420616252230294091406255836554145185951240978356520620177902539293 \cr
$6$ &   0.0201910847096391145836053749493573987118330733550628154025906529456\cr
$7$ &   0.0098027612073790924969564942497359805350512687186310168243371528657\cr
$8$ &   0.0047989747927418270016992324939068507919767494168737399281723990294\cr
$9$ &   0.0023641452020886181354412370986078224447913391845724242369671517436 \cr
$10$ &  0.0011703098147530048368486863035363234141272479119157896271724328508\cr
\hline
\end{tabular}
\end{equation*}
\caption{Upper computed approximation bounds $\alpha_{n,N,+}(s)$ for $E_2$ transfer operator $\l_{s}$ 
with $s=s^-$, $N=600$, and disc $D$
centred at $c\approx 0.758687$, of radius $\rho\approx 0.957589$}
\end{table}

\begin{table}[H]
\begin{equation*}
\begin{tabular}{|r|r|}
\hline
$n$  & $K_{s}^n E_n(h)$ \qquad \qquad\qquad\qquad\qquad\qquad\qquad\qquad \\
\hline 
$26$ &  $1.7205402918728479471042338789554711763326940740466743 \times 10^{-86}$ \cr 
$27$ &  $9.5978010692386084808038394023982841330869065861226330 \times 10^{-94}$ \cr
$28$ &  $2.737417814947540988901740511033648063467122791471394 \times 10^{-101}$ \cr
$29$ &  $3.991837779947558814663544901589857709951099663953540 \times 10^{-109}$ \cr
$30$ &   $2.976234382308236859886112971018657684658758908913873\times 10^{-117}$\cr
$31$ &   $1.134550484615336330129091070266090192517568093692057\times 10^{-125}$\cr
$32$ &   $2.211276104496105402944501365002379392554065222342807\times 10^{-134}$\cr
\hline
\end{tabular}
\end{equation*}
\caption{Euler bounds $K_{s}^n E_n(h)$ (on the $n^{th}$ Taylor coefficient of the determinant 
for the $E_2$ transfer operator $\l_s$) with $s=s^-$}
\end{table}

\bigskip

\begin{table}
\begin{equation*}
\begin{tabular}{|r|r|}
\hline
$n$  & $\beta_{n,N,+}^{M,+}(s)$ \qquad\qquad \qquad \qquad\qquad\qquad\qquad\qquad\qquad\qquad \\
\hline 
$0$ & $1.00000000000000000000000000000000000000000000000000000000000000000000000000$ \cr
$1$ & $1.91923648979580309318951635180234393904884374850026688921303476745864277943$ \cr
$2$ & $1.09811675194206604762230704346732795997751970929683510044470721043734001309$ \cr
$3$ & $0.24618999584155235513565815243210418520583365378089116293710254687241434795$ \cr
$4$ & $0.02398559740297469793182812221795461172137513819594467292973150895628420238$ \cr
$5$ & $0.00106919598571977874103212434018320434942648472790810029433803585219678501$ \cr
$6$ & $0.00002245831360965568426299680358374853210939596804716334173441483413901923$ \cr
$7$ & $2.2642019462375962430662506716612064307152569131758370772288306840900 \times 10^{-7}$ \cr 
$8$ & $1.1092419528871585130899796268449651654078217715387698682639501376708 \times 10^{-9}$ \cr
$9$ &  $2.663650269994059350891751457108890432732071400321264474469330002798 \times 10^{-12}$\cr
$10$ & $3.155171165530321941301909639176345854820087927706592194812388623174 \times 10^{-15}$\cr
$11$ & $1.852432231426985677242256749394660424524281973738655903624698338156 \times 10^{-18}$\cr
$12$ & $5.410594019029701157763137174999660406055719684315663742414740542830 \times 10^{-22}$\cr
$13$ & $7.885051899585888435773423343552379506988548488916647635158418850050 \times 10^{-26}$\cr
$14$ & $5.747100233562844459509048233665882356972216861732638504210303895791 \times 10^{-30}$\cr
$15$ & $2.099041252743632552050904627419516338940376363311264696658378460074 \times 10^{-34}$\cr
$16$ & $3.847903057092197973673777897871275775937411069875824271304861796633 \times 10^{-39}$\cr
$17$ & $3.545294989432407670621821723745739978197914980574557158230527004120 \times 10^{-44}$\cr
$18$ & $1.643668789004361742194939215063268183353658869302130234108066601797 \times 10^{-49}$\cr
$19$ & $3.838399584352345469129330407144020664484419600898330810312866654442 \times 10^{-55}$\cr
$20$ & $4.519027888488147753152792404295333548952705689939991946764902639890 \times 10^{-61}$\cr
$21$ & $2.684346574656834154151019745874691411943090212034757600858330903379 \times 10^{-67}$\cr
$22$ & $8.050690502405021083882671470235302010905655390286812186449358629811 \times 10^{-74}$\cr
$23$ & $1.219830446701270894665408875558427194881417102614891830624858080153 \times 10^{-80}$\cr
$24$ & $9.342902106203197589981798759839115586201690686680609856085682409723 \times 10^{-88}$\cr
$25$ & $3.619108237222286228053279698772494015265793565703855372730270709162 \times 10^{-95}$\cr
$26$ & $7.09350106835309573393504576867864314275082021347185128923856949603 \times 10^{-103}$\cr
$27$ & $7.03791180218706916229135621256991565035863969280596747417493561373 \times 10^{-111}$\cr
$28$ & $3.53607154449140821670269779432007384528678228577107631018236474461 \times 10^{-119}$\cr
\hline
\end{tabular}
\end{equation*}
\caption{Upper computed Taylor bounds $\beta_{n,N,+}^{M,+}(s)$ for $E_2$ transfer operator $\l_{s}$ with $s=s^-$, 
$M=400$, $N=600$, and disc $D$
centred at $c\approx 0.758687$, of radius $\rho\approx 0.957589$}
\end{table}


\begin{thebibliography}{111}

\bibitem{bandtlow} O. F. Bandtlow,  
Resolvent estimates for operators belonging to exponential classes,
{\it Integral Equations Operator Theory}, {\bf 61} (2008), 21--43.

\bibitem{bj}
O. F. Bandtlow \& O. Jenkinson,  
Explicit eigenvalue estimates for transfer operators acting on spaces of holomorphic functions,
{\it Adv. Math.}, {\bf  218} (2008), 902--925. 

\bibitem{bedford}
T. Bedford,
Applications of dynamical systems theory to fractals - a study of
cookie-cutter Cantor sets,
in `Fractal Geometry and Analysis',
ed.~J. B\'elair and S. Dubuc, 1991,
Kluiver.


\bibitem{bowen}
R. Bowen,
Hausdorff dimension of quasicircles,
{\it Inst. Hautes \'Etudes Sci. Publ. Math.},
{\bf 50}  (1979),
11--25.



\bibitem{bumby1}
R. T. Bumby,
Hausdorff dimensions of Cantor sets,
{\it J. Reine Angew. Math.},
{\bf 331}
 (1982), 192--206.


\bibitem{bumby2}
R. T. Bumby,
Hausdorff dimension of sets arising in number theory,
Number theory (New York, 1983--84), Lecture Notes in Math., vol.~1135, Springer,
pp.~1--8,
1985.

\bibitem{cusick1}
T. W.  Cusick, 
Continuants with bounded digits,
{\it Mathematika}, {\bf 24} (1977), 166--172.

\bibitem{cusick2}
T. W.  Cusick, 
Continuants with bounded digits II,
{\it Mathematika}, {\bf 25} (1978), 107--109.

\bibitem{euler}
L. Euler, {\it Introductio in Analysin Infinitorum}, Marcum-Michaelem Bousquet, Lausannae, 1748.


\bibitem{falconer}
K. Falconer
{\it Fractal geometry. Mathematical foundations and
applications}, 
John Wiley \& Sons, Ltd., 1990.

\bibitem{falknussbaum} R. S. Falk \& R. D. Nussbaum,
$C^m$ eigenfunctions of Perron-Frobenius operators and a new approach to numerical computation of Hausdorff dimension,
{\it arXiv preprint, arXiv:1601.06737}.

\bibitem{falknussbaum1} R. S. Falk \& R. D. Nussbaum,
$C^m$  eigenfunctions of Perron-Frobenius operators and a new approach to numerical computation of Hausdorff dimension: applications in $\R^1$,
{\it arXiv preprint, arXiv:1612.00870}

\bibitem{falknussbaum2} R. S. Falk \& R. D. Nussbaum,
$C^m$  eigenfunctions of Perron-Frobenius operators and a new approach to numerical computation of Hausdorff dimension: applications to complex continued fractions,
{\it arXiv preprint, arXiv:1612.00869}

\bibitem{ggk}
I. Gohberg, S. Goldberg \& M. A. Kaashoek,
{\it Classes of linear operators vol.~1}, 1990,
Birkh\"auser, Berlin.

\bibitem{good}
I. J. Good,
The fractional dimensional theory of continued fractions,
{\it Proc. Camb. Phil. Soc.},
{\bf 37} (1941),
199--228.


\bibitem{hensley1989}
D. Hensley,
The Hausdorff dimensions of some continued fraction Cantor sets,
{\it J. Number Theory}, {\bf 33} (1989), 182--198.

\bibitem{hensley1996}
D. Hensley, 
A polynomial time algorithm for the Hausdorff dimension of continued fraction Cantor sets,
{\it J. Number Theory}, {\bf 58} (1996), 9--45.

\bibitem{jarnik}
I. Jarnik, Zur metrischen Theorie der diophantischen Approximationen,
{\it Prace Mat.-Fiz.}, {\bf 36} (1928), 91--106.

\bibitem{jp}
O. Jenkinson \& M. Pollicott,
Computing the dimension of dynamically defined  sets :
 $E_2$ and bounded continued fractions,
 {\it Ergod. Th. \& Dynam. Sys.},
 {\bf 21} (2001), 1429--1445.

\bibitem{kontorovich} A. Kontorovich,
From Apollonius to Zaremba:
local-global phenomena in thin orbits,
{\it Bull. Amer. Math. Soc.},
{\bf 50} (2013), 187--228.

\bibitem{mauldinurbanski}
R. D. Mauldin \& M. Urba\'nski,
{\it Graph directed Markov systems: geometry
and dynamics of limit sets}, 
Cambridge University Press, 2003.

\bibitem{ruelle} D. Ruelle,
Zeta-functions for expanding maps and Anosov flows,
{\it Invent. Math.}, {\bf 34} (1976), 231--242.

\bibitem{shallit}
J. Shallit,
Real numbers with bounded partial quotients: a survey,
{\it Enseign. Math.}, {\bf 38} (1992), 151--187.

\bibitem{shapiro} 
J. H. Shapiro, 
{\it Composition operators and classical function theory},
Springer-Verlag, 1993.

\bibitem{simon}
B. Simon,
{\it Trace ideals and their applications (LMS Lecture Note Series 35)},
Cambridge University Press,
Cambridge,
1979.



\end{thebibliography}
\end{document}